\documentclass[12pt,reqno]{amsart}
\usepackage{graphicx}
\usepackage{amsmath}
\usepackage{amssymb}
\usepackage{amscd}
\usepackage{hhline}
\usepackage{mathrsfs}
 \usepackage{color}

\newtheorem{theorem}{Theorem}

\newtheorem{theoremc}{Theorem}

\newtheorem{rk}[theoremc]{Remark}

\newtheorem{prop}[theorem]{Proposition}
\newcommand\bib[1]{\bibitem[#1]{#1}}

\newcommand\com[1]{}
\newcommand\C{{\mathbb C}}
\newcommand\Cc{{\let\mathcal\mathscr\mathcal C}}

\renewcommand\d{\delta}

\renewcommand\l{\lambda}

\newcommand\oo{\omega}
\newcommand\op[1]{\mathop{\rm #1}\nolimits}
\newcommand\ot{\otimes}
\newcommand\p{\partial}

\newcommand\R{{\mathbb R}}

\newcommand\vp{\varphi}
\newcommand\we{\wedge}

\begin{document}

 \title[Submaximal metric projective and affine structures]{Submaximal metric projective \\ and metric affine structures}
 \author{Boris Kruglikov \& Vladimir Matveev}
 \date{}
 \dedicatory{Dedicated to Mike Eastwood on the occasion of his 60${\,}^{th}$ birthday}
 \address{BK: \ Institute of Mathematics and Statistics, University of Troms\o, Troms\o\ 90-37, Norway.
\quad E-mail: {\tt boris.kruglikov@uit.no}. \newline
 \hphantom{W} VM: \ Mathematisches Institut, Friedrich-Schiller-Universit\"at, 07737, Jena, Germany.
\quad Email: {\tt vladimir.matveev@uni-jena.de}}

 \vspace{-14.5pt}
 \begin{abstract}
We prove that the next possible dimension after the maximal $n^2+2n$ for the
Lie algebra of local projective symmetries of a metric on a manifold of dimension $n>1$
is $n^2-3n+5$ if the signature is Riemannian or $n=2$,
$n^2-3n+6$ if the signature is Lorentzian and $n>2$, and $n^2-3n+8$ elsewise.
We also prove that the Lie algebra of local affine symmetries
of a metric has the same submaximal dimensions (after the maximal $n^2+n$)
unless the signature is Riemannian and $n=3,4$, in which case
the submaximal dimension is $n^2-3n+6$.
 \end{abstract}

 \maketitle

\section*{Introduction}

Consider a linear torsion-free connection $\Gamma=(\Gamma^i_{jk})$ on a smooth connected manifold $M^n$ of dimension
$n\geq2$. A vector field $v$ is called a {\it projective symmetry}, or a {\it projective vector field}, if its local flow sends geodesics (considered as unparameterized curves) to geodesics. Since S.Lie \cite{L} it is known that projective vector fields form a Lie algebra, which we denote by $\mathfrak{p}(\Gamma)$. A vector field $v$ is called an {\it affine symmetry}, or an {\it affine vector field}, if its local flow preserves $\Gamma$; affine vector fields also form a Lie algebra $\mathfrak{a}(\Gamma)$, which we call {\it affine algebra}.
Obviously $\mathfrak{a}(\Gamma)\subseteq\mathfrak{p}(\Gamma)$ is a subalgebra.

It follows from E.Cartan \cite{C} that  $\dim(\mathfrak{p}(\Gamma))\le n^2+2n$ and a connection with the maximal dimension of the projective algebra is {\it projectively flat}, i.e. in a certain local coordinate system the geodesics are straight lines. I.Egorov \cite{E$_1$} proved that the next possible dimension of $\mathfrak{p}(\Gamma)$, the so-called {\it submaximal\/} dimension (this is the maximal dimension
among all non-flat structures), is $n^2-2n+5$ for $n>2$. For $n=2$,  it
was known since S. Lie \cite{L} and A.Tresse \cite{T} that the submaximal dimension is 3.

However, for $n>2$,  the projective structures realizing this dimension are non-metric, in the
sense there exists no (local) metric such that its Levi-Civita connection has
$\dim(\mathfrak{p}) = n^2-2n+5$. This observation follows for example from \cite[(3.5)]{EM},
which can be viewed as  a system of linear equations on the components of the metric $g$,
whose coefficients come from the components of the projective Weyl tensor $W$.
By Egorov \cite{E$_1$}, in a certain coordinate system the connection with the submaximal dimension
of the projective algebra has only two non-zero term $\Gamma_{23}^1=\Gamma_{32}^1=x_2$.
Direct calculation shows that the only nonvanishing components of the  projective Weyl tensor
are  $W_{232}^1=1=-W_{322}^1$; substitution of this into \cite[(3.5)]{EM} yields a system of linear equations such that any solution $g$ is a degenerate symmetric tensor.

Non-metrizability of the Egorov's submaximal projective structure was obtained
independently (and by another method) by S.Casey and M.Dunajski.
In fact, it follows instantly from our first main result
(below $\delta^2_n$ is the Kronecker symbol, i.e. 1 for $n=2$ and 0 else):

 \begin{theorem}\label{Thm1}
Let $\Gamma$ be the Levi-Civita connection of a metric $g$ on $M^n$.
Assume that $\Gamma$ is not projectively flat at least at one point
(i.e. $g$ is not a metric of constant sectional curvature).
Then the maximal possible dimension of the symmetry algebra
$\mathfrak{p}(g)=\mathfrak{p}(\Gamma)$ is equal to
 \begin{itemize}
\item $n^2-3n+5$, \ if $g$ has Riemannian signature,%
\footnote{Within this paper we consider the metrics up to multiplication by $\pm 1$ (since multiplication by a nonzero constant does not change the projective and affine algebras).
In particular both signatures $(+,-,\dots,-)$ and $(-,+,\dots,+)$ are Lorentzian for us,
and we view positively and negatively definite metrics as Riemannian.}
\item $n^2-3n+6-\delta^2_n$, \ if $g$ has Lorentzian signature,
\item $n^2-3n+8$, \ if $g$ has the general signature.
 \end{itemize}
 \end{theorem}
The bound for the general signature was obtained by Mikes \cite{Mi$_1$}.
Our approach however differs from his.

Notice that in the global setting, i.e. if we replace the projective algebra by a projective group,
the sub-maximal projective connection is metric.
The reason is that there are locally-projectively-flat manifolds whose projective group actually has
dimension $n^2+n$ \cite{Y} (this is the global submaximal bound).

If $M$ is closed and the metric $g$ is Riemannian of non-constant sectional curvature, then
the sub-maximal bound is ${n\choose2}+1$ for all $n\ne 4$; for $n=4$ this dimension is ${n\choose2}+2=8$.
Indeed by \cite{M$_1$,M$_2$}, on closed Riemannian manifolds of nonconstant sectional curvature,
the projective group acts by isometries, so the claim follows from \cite{Yan,E$_2$,KN}.
The corresponding models are precisely $\mathbb{S}^1\times\mathbb{S}^{n-1}$ for $n\ne 4$,
possibly quotient by a finite group, and $\mathbb{C}P^2$ for $n=4$.

The problem of determining the dimension gaps (lacunes in the terminology of the Russian geometry school)
between the maximal and sub-maximal structures is classical, see the discussion in \cite{K$_2$}.

\smallskip

Let us now discuss an analogous question for the affine algebra.
The maximal dimension of the space of affine symmetries of an
affine connection is classically known to be $n^2+n$. The submaximal
dimension is $n^2$ and this was also found by I.Egorov \cite{E$_3$}.
The corresponding connections are projectively flat, and for projectively
non-flat connections the affine algebra has maximal dimension $n^2-2n+5$, $n>2$
\cite{E$_1$} (it equals 3 for $n=2$). Again, all these submaximal connections are non-metric.

Our second main result concerns
submaximal dimensions of the affine algebras $\mathfrak{a}(\Gamma)$ of Levi-Civita connections $\Gamma$.

 \begin{theorem}\label{Thm2}
Non-flat metrics $g$ on a manifold $M^n$ have maximal dimension of the affine algebra
$\mathfrak{a}(g)=\mathfrak{a}(\Gamma)$ equal to
 \begin{itemize}
\item $n^2-3n+5+\d_n^3+\d_n^4$, \ if $g$ has Riemannian signature,
\item $n^2-3n+6-\d_n^2$, \ if $g$ has Lorentzian signature,
\item $n^2-3n+8$, \ if $g$ has the general signature.
 \end{itemize}
 \end{theorem}

In the process of the proof of Theorems \ref{Thm1} and \ref{Thm2} we essentially describe
all metrics for the submaximal dimension of the projective
and affine algebras of a metric connection.

Throughout the paper, all our considerations are local (so isomorphisms to the models
are understood locally).

Let us finally specify the gap $\Delta^\mathfrak{p}_1$ between the maximal dimension
of the projective algebra and the submaximal one, and the gap $\Delta^\mathfrak{p}_2$ between the
submaximal projective and submaximal metric projective dimensions
(of un-restricted signature):

 \begin{center}
\begin{tabular}{c||c|c|c|c|c|c|c|c|c}
$n$ & 2 & 3 & 4 & 5 & 6 & 7 & 8 & 9 & \dots\\
\hline
$\Delta^\mathfrak{p}_1$ & 5 & 7 & 11 & 15 & 19 & 23 & 27 & 31 & \dots \\
\hline
$\Delta^\mathfrak{p}_2$ & 0 & 2 & 1 & 2 & 3 & 4 & 5 & 6 & \dots
\end{tabular}
 \end{center}

For affine algebras the corresponding gaps are the following
(with the obvious modification to define $\Delta^\mathfrak{a}_i$)

 \begin{center}
\begin{tabular}{c||c|c|c|c|c|c|c|c|c}
$n$ & 2 & 3 & 4 & 5 & 6 & 7 & 8 & 9 & \dots\\
\hline
$\Delta^\mathfrak{a}_1$ & 2 & 3 & 4 & 5 & 6 & 7 & 8 & 9 & \dots \\
\hline
$\Delta^\mathfrak{a}_2$ & 1 & 3 & 4 & 7 & 10 & 13 & 16 & 19 & \dots
\end{tabular}
 \end{center}

\section{Degree of mobility and useful estimates.}\label{S1}

Two metrics $g$ and $\bar g$ on a manifold $M$ are {\it projectively equivalent},
if any $g$-geodesic is a reparameterization of a $\bar g$-geodesic.
This can be expressed \cite{S$_2$} through the $(1,1)$-tensor
$a={\bar g}^{-1}g\cdot\left|\frac{\det(\bar g)}{\det(g)}\right|^{1/(n+1)}$,
where ${\bar g}^{-1}$ is the inverse of ${\bar g}$ ($\bar g^{ik}\bar g_{kj}=\delta^i_j$),
and ${\bar g}^{-1}g$ is the contraction (${\bar g}^{ik}g_{kj}$):
the metrics $g$ and $\bar g$ are geodesically equivalent if and only if
 \begin{equation}\label{sin}
(n+1)a^i_{j,k}=a^{is}_{\,,s}\,g_{jk}+a^s_{j,s}\d^i_k.
 \end{equation}
In the coordinate-free notation the above formula reads
 $$
(n+1)\nabla a=\op{div}(g^{-1}a)\ot g+\op{div}(a)\ot\op{Id}. 
 $$

Dimension $D(g)$ of the solution space $\op{Sol}(\ref{sin})$ of this linear PDE system
on unknown $a$ is called the {\it degree of mobility\/} of $g$.

Let us denote by  $I(g)$ the algebra of infinitesimal isometries (Killing vector fields),
by $H(g)$ the algebra of homotheties, and by $C(g)$ the algebra of the conformal
vector fields. These are given by the equation $L_v g =\l\cdot g$ on
the vector field $v$, where $\l$ is respectively zero, a constant or an arbitrary function.
Clearly $I(g)\subseteq H(g)\subseteq C(g)$.

In this paper we shall actively use the following two estimates due to
\cite[Theorem 1 and Theorem 2]{Mi$_1$}:
 \begin{eqnarray}
\dim \mathfrak{p}(g)& \le & \dim I(g)+D(g) \label{est1} \\
\dim \mathfrak{p}(g)& \le& \dim H(g)+D(g)-1. \label{est2}
 \end{eqnarray}

Let us prove these estimates (our proof is different from that of Mikes and is much simpler).
It is well known (see \cite{M$_2$} and references therein) that $v\in \mathfrak{p}(g)$ iff
 \begin{equation}\label{Lvg}
a=g^{-1}L_vg-\tfrac1{n+1}\op{Trace}(g^{-1}L_vg)\cdot\op{Id}
 \end{equation}
is a solution of (\ref{sin}). Denote by $\phi:\mathfrak{p}(g)\to \op{Sol}(\ref{sin})$
the linear map sending $v$ to the right hand side of (\ref{Lvg}). Since
$\op{Ker}(\phi)=I(g)$, the rank theorem yields inequality \eqref{est1}.

In order to obtain the second estimate, we observe that
(\ref{sin}) has the obvious one-dimensional subspace of constant solutions
$\R\cdot\op{Id}$. Let $\pi:\op{Sol}(\ref{sin})\to\op{Sol}(\ref{sin})/(\R\cdot\op{Id})$
be the projection. Then the linear map
$\pi\circ\phi:\mathfrak{p}(g)\to\op{Sol}(\ref{sin})/(\R\cdot\op{Id})$
has kernel $H(g)$ and \eqref{est2} follows.

Below we will use the following results on the degree of mobility.
It is always is bounded so:
 $$
D(g)\le {n+2\choose2},
 $$
the equality corresponds to the space of constant curvature \cite{S$_2$}.
The next biggest (submaximal) dimension in any signature is \cite{Mi$_2$}:
 \begin{equation}\label{MiK}
D(g)_{\text{sub.max}}={n-1\choose2}+1.
 \end{equation}
Finally, let us recall from \cite[Lemma 6]{KM} that if the Weyl conformal curvature tensor
of the metric $g$ vanishes, but $g$ is not of constant sectional curvature, then $D(g)\le 2$.

\section{Riemannian case, dimension $n>3$}\label{S2}

In this section we assume that the metric is Riemannian, $n>3$.
Denote by $S_c^n$ the space form of constant curvature $c$, i.e.
the sphere $\mathbb{S}^n\subset\R^{n+1}$ of radius $c>0$, the Euclidean space
$\mathbb{R}^n$ for $c=0$, or the hyperbolic space $\mathbb{H}^n_c$ for $c<0$
equipped with the standard metrics.

Riemannian manifolds with abundant isometries were studied (among others) by
H.C.Wang \cite{W}, K.Yano \cite{Yan}, I.Egorov \cite{E$_2$}, S.Kobayashi and T.Naga\-no \cite{KN}.
According to the local versions of their results, the Riemannian metrics $g$
of non-constant sectional curvature on $M^n$ with $\dim I(g)\ge{n-1\choose 2}+3$, $n>2$,
are contained in the following list (see the Appendix for details):

\smallskip

 \begin{enumerate}
 \item[(a)]
$\dim I(g)={n\choose2}+2=8$, $n=4$. The corresponding $g$ is a K\"ahler metric on a
complex surface with constant nonzero holomorphic sectional curvature (e.g. Fubini-Study
metric on $\C P^2$).
 \item[(b)]
$\dim I(g)={n\choose2}+1$. The corresponding $g$ is the standard metric on the product
$M^n=\R^1\times S^{n-1}_c$ ($c\ne0$).
 \item[(c$_1$)]
$\dim I(g)={n\choose2}$. The corresponding $g=dt^2+a(t)^2ds^2_0$ is the warped product metric,
where $ds_0^2$ is the standard metric on $S^{n-1}_c$ and $a(t)$ is a generic function
(such that $g$ is not of constant sectional curvature and not as in (b)).
 \item[(c$_2$)]
$\dim I(g)={n\choose2}$, $n=6$. $M$ is a K\"ahler manifold of complex dimension 3
of constant nonzero holomorphic sectional curvature.
 \item[(d$_1$)]
$\dim I(g)={n-1\choose2}+3$. The corresponding $M^n=S^2_c\times S^{n-2}_{\bar c}$, and
the constants $c,\bar c$ are not simultaneously zero.
 \item[(d$_2$)]
$\dim I(g)={n-1\choose2}+3$, $n=8$. The corresponding $M$ is a K\"ahler manifold of complex dimension 4
with constant nonzero holomorphic sectional curvature.
 \end{enumerate}

In all these cases, except possibly (c$_1$), $H(g)=I(g)$. Thus
the submaximal dimension of the homothety algebra is $\dim H(g)={n\choose2}+1$
for $n\neq4$. Consequently in the cases (b) and (c$_1$), where the metric $g$ is
conformally flat and so $D(g)\le2$, we obtain from (\ref{est2}):
 $$
\dim\mathfrak{p}(g)\le {n\choose2}+1+1<n^2-3n+5.
 $$

Consider the spaces (a,c$_2$,d$_2$) of nonzero constant holomorphic sectional curvature.
Then $D(g)=1$ (since it is always so for the irreducible symmetric spaces \cite{S$_1$})
and we conclude similarly that $\dim\mathfrak{p}(g)$ is strictly less than the bound
from Theorem \ref{Thm1}.

For the case (d$_1$) we have, combining (\ref{est2}) and (\ref{MiK}):
 $$
\dim \mathfrak{p}(g)\le {n-1\choose2}+3+{n-1\choose2}=n^2-3n+5.
 $$

Finally, if $\dim I(g)<{n-1\choose2}+3$, then $\dim H(g)<{n-1\choose2}+3$.
Indeed, if $I(g)$ acts transitively, then $H(g)=I(g)$ unless the metric is flat
everywhere\footnote{Indeed, if $\vp^*g=\l\cdot g$ for a map $\vp:M\to M$, then
$\vp^*\|R_g\|^2=\l^{-2}\|R_g\|^2$, where $R_g$ is the Riemann curvature tensor
and $\|\!\cdot\!\|$ is the $g$-norm. Thus either $\l=1$ or $R_g(x)=0$ for any fixed point $x$
of $\vp$. Given transitivity of $I(g)$, for any $x$ there is a $\l$-homothety $\vp$ with $\vp(x)=x$,
whence the claim.}. In the intransitive case, $M$ is (locally) foliated by dimension $(n-1)$ leaves,
to which the Killing fields are tangent. On these leaves (restriction to them uniquely
determines the Killing field \cite[Lemma 2.1]{KN}) the metric is not of constant sectional curvature
(otherwise we have the case (c)), and thus
$\dim H(g)\le\dim I(g)+1\le{n-2\choose 2}+2$.
So in this case $\dim\mathfrak{p}(g)$ is strictly less than the bound of Theorem \ref{Thm1}.

 \begin{rk}
An alternative approach to the local version of \cite{KN} is based on the recent
results \cite{DT}, where, by an argument independent from \cite{KN}, it is demonstrated
that a Riemannian conformally non-flat metric $g$ has $\dim C(g)\le \binom{n-1}2+3$
if $n\ne4,6$ (in \cite{KT} this was deduced from the local version of \cite{KN}).
Our arguments apply since $\dim H(g)\le \binom{n}2+1$ for conformally flat $g$ of
non-constant curvature.
 \end{rk}

It follows from the above estimates and the analysis of the obtained models
that the equality for $n>3$ is attained only in the following sub-case of (d$_1$):
 $$
M=S^2_c\times\R^{n-2},\quad c\ne0
 $$
(in the case of the second factor $S^{n-2}_{\bar c}$ having curvature $\bar c\ne0$ all
projective transformations are isometries). The projective transformations of this $M$
consist of 3-dimensional space of isometries of the first factor
(that is $so(3)$ or $sl(2)$ for $c>0$ or $c<0$ respectively)
plus arbitrary affine transformations $x\mapsto A\cdot x+b$ of the second.

\section{Lorentzian signature, dimension $n>3$}\label{S3}

In this section we consider Lorentzian manifolds of dimension $n>3$.
Complete classification of such metrics with the largest dimensions of $I(g)$
is not known to us, so we approach the problem differently.

It is still true that for $g$ of non-constant sectional curvature
$\dim I(g)$ does not exceed ${n\choose2}+1$ in the Lorentzian signature
(in fact, the inequality fails only when $n=4$ for the Riemannian and
split signature \cite{E$_3$}). If the metric $g$ is conformally flat,
but not of constant sectional curvature, then by (\ref{est1}) we get
 $$
\dim\mathfrak{p}(g)\le {n\choose2}+1+2<n^2-3n+6.
 $$

On the other hand, if $(M,g)$ is not conformally flat,
then by \cite{DT} $\dim H(g)\le\dim C(g)\le{n-1\choose2}+4$ and so,
using (\ref{est2}) and (\ref{MiK}), we get
 $$
\dim \mathfrak{p}(g)\le{n-1\choose2}+4+{n-1\choose2}=n^2-3n+6.
 $$
This estimate is achieved on the Lorentzian pp-wave metric, trivially extended to dimension $n$,
$M^n=M^3_{pp}(x,y,z)\times\R^{n-3}(u_4,\dots,u_n)$, see \cite{KT}:
 $$
g=2dx\,dy+z^2dy^2+dz^2+\sum_{i=4}^n du_i^2.
 $$
This metric has $\frac12(n^2-3n+8)$ Killing vector fields
 $$
\p_x,\ \p_y,\ e^y(\p_z-z\,\p_x),\ e^{-y}(\p_z+z\,\p_x),\
\p_{u_i},\ u_i\p_x-y\,\p_{u_i},\ u_i\p_{u_j}-u_j\p_{u_i},\
 $$
1 pure homothety
 $$
2x\,\p_x+z\,\p_z+\sum_{i=4}^nu_i\p_{u_i}
 $$
and $\frac12(n^2-3n+2)$ pure affine fields
 $$
y\,\p_x,\ u_i\p_x,\ u_i\p_{u_j}+u_j\p_{u_i}
 $$
(the latter is easy to check as the connection is rather simple:
$\Gamma_{23}^1=\Gamma_{32}^1=2z$, $\Gamma_{22}^3=-z$ and the other Christoffel symbols are zero).
There are no non-affine projective fields.
Thus the totality of the linearly independent
projective symmetries is $\dim\mathfrak{p}(g)=n^2-3n+6$. This proves the claim for $n>3$.

Let us note that the submaximal model described above is (visibly) not unique.
We can take any metric from the list of Kru\v{c}kovi\v{c} \cite{Kr} (see Section \ref{S4.5})
that has 4 Killing fields and 1 essential homothety/affine field, and extend it
trivially to $n$ dimensions, achieving the same result: The new metric $g$ will
have $(n-2)$-dimensional space of parallel vector fields, yielding $(n-1)(n-2)$ affine symmetries,
to which we add the 4 fields coming from the 3D metric (3 Killing fields + 1 homothety, as one Killing
field that is parallel was already counted). All maximal models are obtained in this way (this
observation is based on the fact \cite[Theorem 5]{FM} that parallel (0,2) tensors are
linear combinations of symmetric products of parallel vectors).

Finally we remark that $\mathfrak{p}(g)=\mathfrak{a}(g)$ for all Lorenzian metrics $g$ of submaximal dimension of the projective algebra. Indeed, if there was an essentially projective symmetry,
then by \cite[Theorem 3]{FM}
 \begin{equation}\label{FeM}
\dim\mathfrak{p}(g)\le\dim I(g)+D(g)-1.
 \end{equation}
According to \cite{DT} for a conformally non-flat metric $g$ $\dim C(g)\le{n-1\choose2}+4$.
It also follows from the construction of loc.cit.\ that at any point $x\in M$ the grading 0 component
$\mathfrak{s}_0(x)$ of the (associated graded to
the naturally filtered) symmetry group $\mathfrak{s}$ satisfies:
$\mathfrak{so}(1,n-1)\not\supset\mathfrak{s}_0\subset\mathfrak{co}(1,n-1)$.
Thus $I(g)\neq C(g)$\footnote{Consider $v\in C(g)$ that satisfies $v_x=0$, $d_xv\in\mathfrak{s}_0(x)\setminus\mathfrak{so}(1,n-1)$.
Then $L_vg(x)=\l\,g_x$ for $\l\neq1$, and so $v$ is not an infinitesimal isometry.}
and so $\dim I(g)\le{n-1\choose2}+3$. Combining this estimate, (\ref{FeM}) and (\ref{MiK})
we obtain $\dim\mathfrak{p}(g)\le n^2-3n+5$,
which contradicts submaximality of the projective algebra.

\section{The proof for the general signature}\label{S4}

Consider now the metric of signature $(p,q)$, where both $p,q\ge2$, so $n=p+q\ge4$.
If the metric $g$ is conformally flat and not of constant sectional curvature, then
by (\ref{est1}) we get
 $$
\dim\mathfrak{p}(g)\le {n\choose2}+2+2<n^2-3n+8.
 $$

Next, by the results of \cite{KT}, for non-conformally flat metric structure we have
 $$
\dim C(g)\le{n-1\choose2}+6.
 $$
and so, using $H(g)\subseteq C(g)$, by (\ref{est1}) and (\ref{MiK}) we get
 $$
\dim\mathfrak{p}(g)\le {n-1\choose2}+6+{n-1\choose2}=n^2-3n+8.
 $$

The equality is attained on the metric of the split signature pp-waves
trivially extended from 4 to $n$ dimensions:
 $$
g_{pp}=dx\,dw+dy\,dz+y^2\,dw^2+\sum_{i=5}^n\epsilon_i\,du_i^2.
 $$
This metric has conformal Weyl curvature tensor $CW=(dy\we dw)^2$,
and is Einstein (Ricci-flat). Moreover the projective symmetries coincide with
its affine symmetries. We have $C(g_{pp})=H(g_{pp})$, and the generators of this algebra
were calculated in \cite{KT}:
 \begin{gather*}
\p_x,\ \p_z,\ \p_w,\ \p_y-2\,yw\,\p_x+w^2\p_z,\ y\,\p_x-w\,\p_z,\\
(z+yw^2)\,\p_x-w\,\p_y-\tfrac13w^3\p_z,\ x\,\p_z-y\,\p_w+\tfrac23y^3\p_x,\\
x\,\p_x+y\,\p_y-z\,\p_z-w\,\p_w,\ 2\,x\,\p_x+y\,\p_y+z\,\p_z,\\
\p_{u_i},\ \epsilon_i u_i\,\p_{u_j}-\epsilon_j u_j\,\p_{u_i},\
2\epsilon_i u_i\,\p_z-y\,\p_{u_i},\ 2\epsilon_i u_i\,\p_x-w\,\p_{u_i}.
 \end{gather*}
In addition, $g_{pp}$ has the following genuine affine symmetries (not homotheties)
 \begin{gather*}
y\,\p_z,\ w\,\p_x,\ y\,\p_x+w\,\p_z,\ 2\epsilon_i u_i\,\p_z+y\,\p_{u_i},\\
2\epsilon_i u_i\,\p_x+w\,\p_{u_i},\ \epsilon_i u_i\,\p_{u_j}+\epsilon_j u_j\,\p_{u_i}.
 \end{gather*}
Thus the total number of linearly independent projective symmetries is
$\op{dim}\mathfrak{p}(g)=n^2-3n+8$ as required in Theorem \ref{Thm1}.

\section{Dimension $n=3$}\label{S4.5}

Consider the special case $n=3$, where $g$ necessarily has Riemannian or Lorentzian signature.

In this dimension the Weyl conformal curvature vanishes identically,
and so (even for non-conformally flat metrics) $D(g)\le 2$, see \cite{KM}. Since the
submaximal $\dim I(g)\le4$ we get by (\ref{est1}): $\dim\mathfrak{p}(g)\le 6$.

On the other hand, in the Riemannian case\footnote{Recall from Section \ref{S2} that $H(g)=I(g)$
if $I(g)$ is transitive. It is easy to check that $\dim I(g)\le3$ if $I(g)$ is intransitive.}
$H(g)\le4$, so (\ref{est1}) implies
$\dim\mathfrak{p}(g)\le 5$. Also if $g$ is not conformally flat, then $\dim C(g)\le4$
by \cite{KT}. Henceforth we get the bound $\dim\mathfrak{p}(g)\le 5$ in this case too.

The local metrics in 3 dimensions with $\dim I(g)=4$ were classified by
G.Kru\v{c}kovi\v{c} \cite{Kr}. There are 8 different cases. The first three are
the Lorentzian metrics:
 \begin{enumerate}
\item $g_1=k\,dx^2+2(2-c)e^{cx}\,dx\,dy+e^{2x}dz^2$,\quad ($c\ne2$)
\item $g_2=k\,dx^2+e^{2x}(2\,dx\,dy-dz^2)$,
\item $g_3=k\,dx^2+e^{x\,\sqrt{4-\omega^2}}\bigl(2\,dx\,dy-
    \frac4{\omega^2}\cos^2(\frac{\omega x}2)\,dz^2\bigr)$.
 \end{enumerate}
For all these metrics $\dim I(g)=4$, $\dim H(g)=5$, $D(g)=2$, so both estimates
(\ref{est1}) and (\ref{est2}) state $\dim\mathfrak{p}(g)\le 6$.

And in fact, this bound is achieved for all 3 cases.
The infinitesimal automorphisms can be shown explicitly.
For instance, for the first metric $g_1$ ($c\ne2$)
the algebra of projective symmetries has generators
 $$
\p_y,\ \p_z,\ z\,\p_y+{\rm e}^{x(c-2)}\p_z,\ \p_x-cy\,\p_y-z\,\p_z,\
\left(2\,y+{\frac{k\,{\rm e}^{-cx}}{(c-2)c}}\right)\p_y+z\,\p_z,\ {\rm e}^{cx}\p_y.
 $$
The first four are Killing fields, the fifth is a homothety and the last is
a projective field for the metric $g$ (in fact it is an affine field for $g$).

For $c=0$ the homothety has to be changed to $(2y+\frac{k}2x)\,\p_y+z\,\p_z$,
and the genuine projective field has to be changed to $2y\,\p_y+z\,\p_z$.

These formulae were obtained using the \emph{DifferentialGeometry} package of \textsc{Maple},
and can be easily verified by hand.
The formulae for $g_2$ and $g_3$ are obtained similarly.

The other metrics (4)-(8) in \cite{Kr} can be of both possible signatures.

The metrics of cases (4), (5) and (6) in the list of Kru\v{c}kovi\v{c} are not
conformally flat, and for them direct calculation yields $H(g)=I(g)$, $D(g)=1$.
Thus by (\ref{est2}) we get $4\le \dim\mathfrak{p}(g)\le 4$, i.e. all projective
transformations in these cases are isometries,
and the metrics are not of submaximal projective symmetry.

The metrics in cases (7) and (8) of loc.cit. are conformally flat, but for them
$H(g)=I(g)$, $D(g)=2$, whence $\dim\mathfrak{p}(g)\le 5$. In fact, in these cases the manifold
is locally $M^3=\R^1\times S^2_c$, where $c\ne0$. There is one genuine affine symmetry
(scaling along $\R^1$), so we have $\dim\mathfrak{p}(g)=5$.
Again for the Lorentzian signature these $M^3$ are not submaximal, but for Riemannian
signature they are.
This finishes investigation of the 3-dimensional case.

\section{Dimension $n=2$}\label{S5}

Dimension 2 is another exception to the above arguments.
Again here the problem is classical: the 2D projective structures
were studied by S. Lie \cite{L}, R. Liouville \cite{Li} and A. Tresse \cite{T}.
In the latter reference it was proven that the submaximal dimension of the symmetry
algebra is 3 and the submaximal projective structures were classified.

The two projective structures arising in this way (see also \cite{Ma}),
when written as the 2nd order ODE on the plane, are
 \begin{equation}\label{LLT}
x\,y''=\epsilon(y')^3-\tfrac12\,y'\quad (\epsilon=\pm1).
 \end{equation}
It is easy to check they are metrizable; the corresponding metrics
 are
 \begin{equation}\label{2D}
g=x\,dx^2-2\,\epsilon x\,dy^2.
 \end{equation}
The projective symmetry algebra is $sl(2)$ realized on the plane $\R^2(x,y)$
via the vector fields $\p_y,\, x\,\p_x+y\,\p_y,\, 2xy\,\p_x+y^2\p_y$.

Another form of (\ref{LLT}) was considered in \cite{BMM}:
$y''=\epsilon\,e^{-2x}(y')^3+\frac12y'$.
It is obtained from (\ref{LLT}) by the transformation $x\mapsto-\epsilon e^x$.

Yet in \cite{K$_1$} this projective connection was written
differently:
 $$
y''=\epsilon(xy'-y)^3.
 $$
The symmetry algebra $sl(2)$ in this realization has
the standard linear representation on $\R^2(x,y)$: $x\,\p_x-y\,\p_y,\ x\,\p_y,\ y\,\p_x$.
A metric corresponding to this projective connection is
 $$
g=\Bigl(\frac{dx}{y^2}-\frac{x\,dy}{y^3}\Bigr)^2-\epsilon\,\frac{dy^2}{y^8}.
 $$
The complete list of the corresponding metrics is contained in \cite{BMM}.

This finishes the proof of Theorem \ref{Thm1}.

\section{Affine symmetries of a metric}\label{S6}

In this Section we prove Theorem \ref{Thm2}, estimating $\mathfrak{a}(g)\subseteq\mathfrak{p}(g)$.

Recall that (for any $n\ge2$) if the Levi-Civita connection $\Gamma$ of the metric $g$ is projectively
flat, then $g$ has constant curvature. This is a variant of Beltrami's theorem, see \cite{EM}.
Thus if $g$ is not of constant sectional curvature, we get the bound
 \begin{equation}\label{bound1}
\dim\mathfrak{a}(g)\le\dim\mathfrak{p}(g)\le n^2-3n+\sigma,
 \end{equation}
where $\sigma=5$, $6$ or $8$ is the same number, depending on the signature of $g$ and dimension $n$ of $M$,
as in Theorem \ref{Thm1}.

We have to show that this bound is achieved within the class of metrics of
nonconstant sectional curvature. But indeed, the submaximal
projective symmetry algebras of the models, studied in the preceding sections
consisted of affine fields only, provided $n>2$. Thus we conclude that the bound is actually sharp
in the considered class.

If $g$ has nonzero constant sectional curvature (in this case
$\dim\mathfrak{a}(g)$ is also less than the maximal value $n^2+n$),
then $\mathfrak{a}(g)=I(g)$ and we get
 \begin{equation}\label{bound2}
\dim\mathfrak{a}(g)={n+1\choose2}.
 \end{equation}
For $n\ge5$ and for $n=2$ this does not exceed the first bound (and is strictly less than it for $n>5$)
in the Riemannian case, and it never exceeds it (for $n>4$, is strictly less than it)
for the Lorentzian signature. For the Riemannian signature and $n=3,4$
the bound in (\ref{bound1}) exceeds the bound in (\ref{bound2}) by 1. Thus
in this case the constant sectional curvature spaces give the submaximal dimension of $\mathfrak{p}(g)$.

In the general signature $(p,q)$, $p,q\ge2$, $n=p+q\ge4$, the bound of (\ref{bound2}) is strictly less
than that of (\ref{bound1}), thus nothing new is added here.

Dimension 2 is again special. If the affine connection is not flat, then $\mathfrak{a}(g)=H(g)$.
Indeed, since the curvature is nonzero, the holonomy algebra is full.
The maximal dimension of $H(g)$ in the case of non-constant scalar curvature is 2
(this is achieved, for example, in the case of metric (\ref{2D}) via its first two symmetries).
Consequently the submaximal dimension of $\mathfrak{a}(g)$ is 3 and is achieved for the
round sphere, the Lobachevsky plane or de Sitter metric. Herewith the claim is proved.

\appendix
\section{Local version of Kobayashi-Nagano 1972}\label{A}

In \cite{KN} the Riemannian manifolds $(M^n,g)$ with the abundant groups $G$ of
isometries (namely those with $\dim G\ge\binom{n-1}2+2$) were classified.
We explain here that their results hold true locally.

 \begin{theorem}\label{thm3}
If the algebra $I(g)$ of symmetries of a Riemannian metric $g$ has dimension
$\dim I(g)\ge\binom{n-1}2+3$, then $(M,g)$ is locally isomorphic to one
of the models from the list (a)-(d) indicated in Section \ref{S2}.
 \end{theorem}

The main difference of the local and global approaches is the usage of group methods,
and these in \cite{KN} enter via the stabilizer $G_x$ of the point $x\in M$
(see Lemmata 2.4 and 2.8 in loc.cit.). To adapt the notations to the local setting and
establish their properties consider the orthogonal frame bundle $\pi:\mathcal{G}\to M$ (= principal
$SO(n)$-bundle) with the Cartan form $\omega\in\Omega^1(\mathcal{G},E(n))$, $E(n)=SO(n)\ltimes\R^n$,
giving the absolute parallelism. We will write $F_x=\pi^{-1}(x)\simeq SO(T_xM)$.

Let $\mathcal{O}^M_x$ denote the germ of neighborhoods of $x$ in $M$ and
$\mathcal{O}^{\mathcal{G}}_{F_x}$ the germ of neighborhoods of $F_x$ in $\mathcal{G}$. Denote
 $$
I_x=\{v\in\mathfrak{D}(\mathcal{O}_x)\,|\,L_v(g)=0,v(x)=0\}
 $$
the $x$-stalk of the sheaf of the symmetries of $g$ and
 $$
G_x=\{\phi:\mathcal{O}^{\mathcal{G}}_{F_x}\to\mathcal{O}^{\mathcal{G}}_{F_x}\,|\,\phi^*\oo=\oo,
\phi(F_x)=F_x\}
 $$
the $x$-stalk of the sheaf of the local isometries of $g$.
We can suppose that all the transformations $\phi$ are defined on the same neighborhood $U_\epsilon=\pi^{-1}(B_\epsilon(x))$, where $B_\epsilon(x)$ is the ball of small radius $\epsilon$
around $x\in M$ with respect to the metric $g$;
then $\phi:U_\epsilon\to U_\epsilon$.
Notice that $\phi_0:\mathcal{O}_x\to \mathcal{O}_x$ given by $\pi\circ\phi=\phi_0\circ\pi$
satisfies $\phi_0^*g=g$.

Since an isometry is uniquely determined by its 1-jet, the topology on $G_x$ is induced
from embedding into the space $J^1_x(M,M)$. Similarly, $I_x$ has the natural topology
of a finite-dimensional vector space.

The following is a local version of the Myers-Steenrod theorem.

 \begin{prop}
The connected component $G_x^0$ of unity in $G_x$ is isomorphic to the closed subgroup in $SO(T_xM)$
generated by $\exp(v)$, $v\in I_x(g)$.
 \end{prop}

 \begin{proof}
Following the proof of Theorem 3.2 \cite{Ko} we establish that $G_x$ is the compact Lie
group defined by the property that $\phi\in G_x$ iff $\phi\cdot\exp(\xi)=\exp(\xi)\cdot\phi$ for
any parallel vector field w.r.t.\ the absolute parallelism, $\omega(\xi)=\op{const}$.
The map $\phi\mapsto\phi(u)$, $u\in F_x$ is an embedding of $G_x$ into $F_x$ and its image
is a closed (hence compact) submanifold.
Thus we can identify $G_x$ and also $G_x^0$ as Lie subgroups in $SO(T_xM)$.

Denote the embedding by $\iota$.
By Palais theorem \cite[Theorem 3.1]{Ko} $I_x$ coincides with the Lie algebra of $\iota(G_x)$.

Indeed, the vector fields $v\in I_x$ lift naturally to vector fields in $\mathcal{G}$ tangent to $F_x$. Denote by $\hat v$ the lifted field $v$ restricted to $F_x$.
The homomorphism $I_x\to\op{Lie}(\iota(G_x))\subset\mathfrak{D}(F_x)$, $v\mapsto\hat v$, is injective.
We claim that it is surjective. Indeed, if $\dim I_x<\dim G_x$, then there exists a 1-parameter
group $\exp(tw)\in G_x$, $w\in\op{Lie}(G_x)\setminus I_x$, with the property $\exp(tw)^*g=g$
(here we identify the vector field $w$ on $\mathcal{O}_x$ and its lift to $\mathcal{G}_{F_x}$),
whence $L_wg=0$ and so $w\in I_x$ -- a contradiction.

Thus $\dim I_x=\dim G_x$ and so the group generated by $\exp(\hat v)$, $v\in I_x$,
coincides with the (connected) Lie subgroup $\iota(G_x^0)\subset SO(T_xM)$.
 \end{proof}

 \begin{rk}
A similar statement holds for any Cartan geometry of type $(G,H)$ with compact $H$.
 \end{rk}

Now rest of the exposition in \cite{KN} can be easily adapted to the local setting:
Every time the stabilizer $G_x$ occurs, it should be understood in the above sense.
The group $\iota(G_x^0)$ act by isometries on the unit sphere $S^{n-1}\subset T_xM$ and the
basic reduction argument of loc.cit.\ goes through.

The group orbits in $M$ locally have to be understood as germs of submanifolds,
and consideration of various discrete subgroups can be ignored.
No other global arguments is contained in \cite{KN}.
This finishes the proof of Theorem \ref{thm3}.

\bigskip

\textsc{Acknowledgment.} We thank Boris Doubrov and Dennis The for
helpful discussions on the issues of the paper \cite{KN}, and for providing us the
early version of \cite{DT}.


\end{document}